\documentclass[conference]{IEEEtran}
\IEEEoverridecommandlockouts

\usepackage{hyperref}

\usepackage{cite}
\usepackage{amsmath,amssymb,amsfonts}
\usepackage{amsmath,bm}
\usepackage{algorithmic}
\usepackage{graphicx}
\usepackage{subcaption}
\usepackage{textcomp}
\usepackage{xcolor}
\newcommand{\R}{\mathbb{R}}
\newcommand{\N}{\mathbb{N}}

  {\end{list}}

\makeatletter

  \gdef\listctr{list\romannumeral\the\@listdepth}\expandafter
  
\makeatother

\newcommand{\ve}[1]{\mathbf{#1}}
\def\x{\ve{x}}
\def\y{\ve{y}}

\def\z{\ve{z}}
\def\e{\ve{e}}

\def\a{\ve{a}}

\def\A{\ve{A}}
\def\Cc{\mathcal{C}}

\def\R{\mathbb R}

\def\prox{{\mbox{prox}}}

\def\prox{\mathrm{prox}} 
\newcommand{\argmin}{\operatornamewithlimits{argmin}}

\newcommand{\BR}{B-rex}
\newtheorem{theorem}{Theorem}
\newtheorem{proposition}[theorem]{Proposition}

\newtheorem{remark}{Remark}

\newtheorem{assump}{Assumption}

\def\BibTeX{{\rm B\kern-.05em{\sc i\kern-.025em b}\kern-.08em
    T\kern-.1667em\lower.7ex\hbox{E}\kern-.125emX}}
\begin{document}

\title{Box-constrained $\ell_0$ Bregman-relaxations\\
\thanks{ME and ES  acknowledge the financial support of the ANR EROSION (ANR-22-CE48-0004). LC acknowledges the financial support of the European Research Council (grant MALIN, 101117133).}
}

\author{\IEEEauthorblockN{Mhamed Essafri}
\IEEEauthorblockA{\textit{IRIT, Université de Toulouse, INP}\\
Toulouse, France \\
mhamed.essafri@irit.fr}
\and
\IEEEauthorblockN{
Luca Calatroni}
\IEEEauthorblockA{\textit{MaLGa, DIBRIS,} 
\textit{Università di Genova},\\
 \textit{MMS, Istituto Italiano di Tecnologia}, \\
Genoa, Italy,\\
luca.calatroni@unige.it}

\and
\IEEEauthorblockN{Emmanuel Soubies}
\IEEEauthorblockA{\textit{IRIT, Université de Toulouse, CNRS}\\
Toulouse, France  \\
emmanuel.soubies@cnrs.fr}
}

\maketitle

\begin{abstract}
Regularization using the $\ell_0$ pseudo-norm is a common approach to promote sparsity, with widespread applications in machine learning and signal processing.
However, solving such problems is known to be NP-hard. Recently, the $\ell_0$ Bregman relaxation (\BR{}) has been introduced as a continuous, non-convex approximation of the $\ell_0$ pseudo-norm. Replacing the $\ell_0$ term with \BR{} leads to exact continuous relaxations that preserve the global optimum while simplifying the optimization landscape, making non-convex problems more tractable for algorithmic approaches.
In this paper, we focus on box-constrained exact continuous Bregman relaxations of $\ell_0$-regularized criteria with general data terms, including least-squares, logistic regression, and Kullback-Leibler fidelities. Experimental results on synthetic data, compared with  Branch-and-Bound methods, demonstrate the effectiveness of the proposed relaxations.
\end{abstract}


\section{Introduction}

Given a possibly undetermined ($M \leq  N$) forward matrix $\A \in \R^{M \times N}$ and  vector of observations $\y \in \R^M$, we consider problems of the form
\begin{equation}\label{eq:problem_setting}
\hat{\x} \in \argmin_{\x \in [l,u]^N} \left\{ \mkern-3mu J_0(\x) := F_\y(\A \x) + \lambda_0 \|\x\|_0 + \frac{\lambda_2}{2} \|\x\|^2_2\right\}
\end{equation}
where $ l \in \R_{\leq 0} \cup \{-\infty \}$ and $u \in \R_{\geq 0} \cup \{+\infty\}$ {define a box constraint}, the terms  $\|\cdot\|_2$ and $\|\cdot\|_0$ denotes respectively  the squared  $\ell_2$ norm and the $\ell_0$ pseudo-norm that counts the number of non-zero elements in a vector of $\R^N$. The hyperparameters  $\lambda_0 > 0$ and $\lambda_2 \geq 0 $ control respectively the sparsity and the $\ell_2$ ridge regularization strengths. 
 Finally,  $F_\y \; :  \; \R^M \mapsto \R_{\geq 0}$ is a data-fidelity function that measures the discrepancy between the model $\A\x$ and the data $\y$, and satisfies the following assumption.
\begin{assump}
The data fidelity function is coordinate-wise separable. i.e, $F_\y(\z) = \sum_{m=1}^M f(z_m;y_m)$, where for each $y \in \R$, $f(\cdot;y)$ is convex, proper, twice differentiable on $(l,u)$ and bounded from below.
\end{assump}

 Beyond the popular least-squares function,  exemplar instances of such data-fidelity terms include the Kullback-Leibler divergence~\cite{cameron2013regression} and logistic loss~\cite{buhlmann2011statistics}, which arise in signal/image processing and machine learning applications.

\begin{remark} To simplify the presentation, we consider in~\eqref{eq:problem_setting} the case where $l:=l_1 = l_2=\dots=l_N$ and $u:=u_1 = u_2=\cdots=u_N$, i.e.~the same box constraint is applied to each component. Note, however, that our results can be easily extended to the general case.
\end{remark}

\subsection{Related Works}

While the $\ell_0$ pseudo-norm is the most natural choice for enforcing sparsity, its discontinuity, non-convexity, and non-smoothness make the associated problem~\eqref{eq:problem_setting} NP-hard~\cite{Nguyen2019}. {Yet there exists a vast literature related to this problem}. One approach is to address the original problem directly, such as with the Iterative Hard Thresholding (IHT), which can be extended to this setting and guarantees convergence to a critical point~\cite{Attouch2013}. In~\cite{Beck-SPS}, the authors studied the proximal mapping of the $\ell_0$ function over symmetric sets that satisfy a submodularity-like property (SOM) and developed algorithms that converge to critical points. For moderately sized problems, branch-and-bound (BnB) methods offer exact solutions at a reasonable computational cost~\cite{Diego,guyard:24}. 

In this work, we focus on exact relaxation approaches, which replace the $\ell_0$ pseudo-norm with a continuous (non-convex) penalty function while preserving global minimizers. Moreover, such relaxed formulations remove some local minimizers of the initial problem, making the optimization landscape more favorable to optimization algorithms. In this context, the authors in~\cite{Weigeneral} proposed an exact relaxation using a capped-$\ell_1$ penalty for cases where $F_\y(\A\cdot) + \frac{\lambda_2}{2} \|\cdot\|_2^2$ is convex, Lipschitz continuous, and non-smooth, applicable to both unconstrained and box-constrained cases. Building on this, they developed a smoothing proximal gradient (SPG) algorithm to find a stationary point of the relaxed problem, which corresponds to a local minimizer of the original problem. However, for the previously mentioned data-terms, $F_\y(\A\cdot) + \frac{\lambda_2}{2} \|\cdot\|_2^2$ is not necessarily Lipschitz continuous. For this setting, quadratic envelopes of the $\ell_0$ pseudo-norm~\cite{carlsson2019convex}, including the continuous exact $\ell_0$ (CEL0) penalty~\cite{soubies2015continuous}, enable exact relaxation for least-squares data terms. Recently, these ideas have been generalized with the introduction of the $\ell_0$ Bregman relaxation (\BR{}), providing exact relaxations for general (non-quadratic) data terms~\cite{Exact2024,Essafri2024MLSP}. Additionally, the work in~\cite{Lazzaretti2021} proposed a weighted-CEL0 relaxation for weighted-$\ell_2$ data terms, offering an approximation of the KL divergence. While these approaches are limited to the unconstrained case, in this paper, we extend the works of~\cite{Exact2024,Essafri2024MLSP} to box-constrained problems.

\subsection{Contributions and Outline}
In this work we extend the framework of exact continuous $\ell_0$ Bregman-relaxations, originally proposed  for unconstrained and non-negatively constrained problems~\cite{Exact2024,Essafri2024MLSP}, to box-constrained problems of the form \eqref{eq:problem_setting}. More precisely, we provide in Proposition \ref{th:exact_relax} an exact relaxation result. In view of designing effective numerical schemes solving the relaxed problem, we consider in Section \ref{sec:minminiz_JPsi} some proximal-based schemes showing, in particular, how the proximal operator of the proposed exact penalty can be efficiently computed. In addition, we will show that the relaxed problem can be minimized via iteratively reweighted $\ell_1$. In Section~\ref{sec:experiments}, we compare the minimization of the relaxed criteria with both BnB and IHT procedures showing good agreement with certified global procedures in the case of small-size problems and the applicability of the approach to larger-scale problems.


\subsection{Notations}

Let $[N] = \{1, 2, \dots, N\}$ denote the set of indices up to $N$. The symbol $\mathbf{0} \in \R^N$ represents the vector of all zeros, and for each $n \in [N]$, $\e_n \in \R^N$ denotes the unit vector of the standard basis in $\R^N$. The set $\R_{\geq 0} = \{x \in \R \mid x \geq 0\}$ represents the non-negative real line. For $n \in [N]$, $\x^{(n)} = (x_1, \dots, x_{n-1}, 0, x_{n+1}, \dots, x_N)$ denotes the vector $\x$ with the $n$-th component replaced by 0.

\section{Box-Constrained \BR{}}

In this section, we extend the $\ell_0$ Bregman Relaxation (\BR {}) introduced in\cite{Exact2024} to the box-constrained Problem~\eqref{eq:problem_setting}. Consider a family $\Psi = \left\lbrace \psi_n \right\rbrace_{n \in [N]}$ of strictly convex, proper, and twice-differentiable functions such that $\mathrm{dom}(\psi_n) \in \{\R,\R_{\geq 0}\}$. We define the box-constrained \BR{} as $B_\Psi^{l,u} : [l,u]^N \to \R_{\geq 0}$ such that for all  $\x \in [l,u]^N$:
\begin{equation}\label{eq:continuous_form} 
\begin{array}{rl}
B^{l,u}_\Psi(\x) = & \displaystyle  \sup_{\alpha \in \R} \sup_{\z \in \mathrm{dom}(\Psi)} \alpha - D_\Psi(\x, \z) \\
& \displaystyle \text{ s.t. }  \alpha - D_\Psi(\cdot, \z) \leq \lambda_0\|\cdot\|_0 + \mathcal{I}_{[l,u]^N}
\end{array}
\end{equation}
where $\mathcal{I}_{[l,u]^N}(\x) = \{ 0 \text{ if } \x \in [l,u]^N ; + \infty \text{ otherwise}\}$ and $D_\Psi$ denotes the Bregman divergence associated to $\Psi$. It is defined, for all $\x,\z \in \mathrm{dom}(\Psi)$ by
\begin{equation}
    D_\Psi(\mathbf{x},\mathbf{z}) = \sum_{n=1}^N d_{\psi_n}(x_n,z_n) \\
\end{equation}
with $d_{\psi_n}(x,z) = \psi_n(x) - \psi_n(z) - \psi_n'(z) (x-z)\; \forall x, z \in \mathrm{dom}(\psi_n)$. Standard choices of the functions $\psi_n$ are the $p$-power functions, the Shannon entropy and the Kullback-Leibler divergence, see~\cite{Exact2024,Essafri2024MLSP}. Note, that as opposed to the unconstrained \BR{} proposed therein, the extension~\eqref{eq:continuous_form} depends on the box-constraint through the term $\mathcal{I}_{[l,u]^N}$.

\begin{figure*}[!ht]
    \centering
    \begin{subfigure}[b]{0.43\linewidth}
        \centering
        \includegraphics[width=\linewidth]{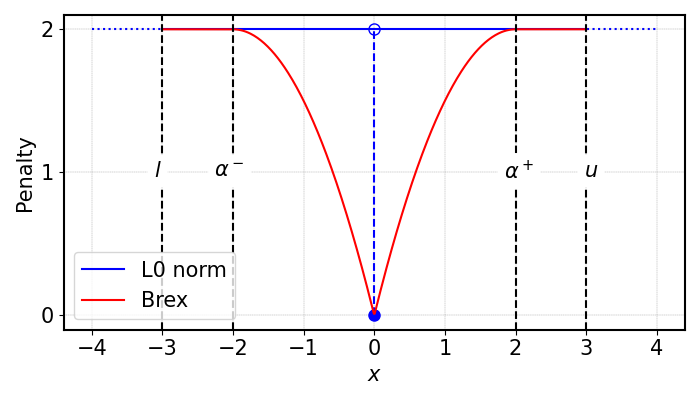}
    \end{subfigure}
    \hspace{0.3cm}
    \begin{subfigure}[b]{0.43\linewidth}
        \centering
        \includegraphics[width=\linewidth]{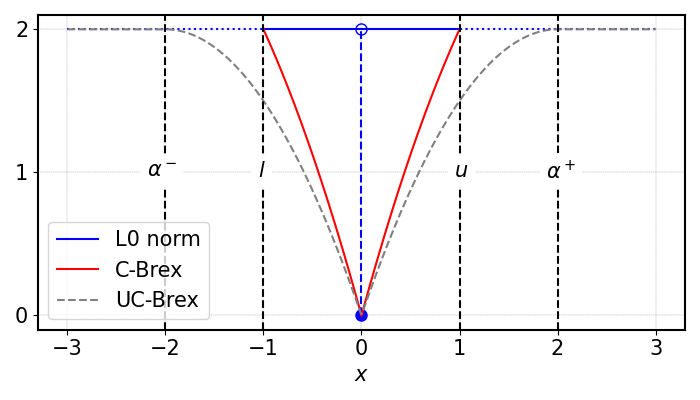}
    \end{subfigure}
    \caption{Illustration of box-constrained \BR{}  when (left) $[l,u] \supseteq [\alpha^-,\alpha^+]$  and (right) $[l,u] \subseteq [\alpha^-,\alpha^+]$. \label{fig:illustration_BR}}
\end{figure*}

\begin{proposition}[Closed form expression of $B_\Psi^{l,u}$]\label{prop:main_prop}
    For all $n \in [N]$, let $\alpha^-_n \leq 0$ and $\alpha^+_n \geq 0$ be such that $[\alpha_n^-, \alpha_n^+]$ defines the set of sublevels $\lambda_0$ of $d_{\psi_n}(0, \cdot)$. Then, for every $\x \in [l,u]^N$, we have
    \(
    B_\Psi^{l,u}(\x) = \sum_{n=1}^N \beta^{l,u}_{\psi_n}(x_n),
    \)
    where, for $x \in [l,u]$, the functions $\beta_{\psi_n}^{l,u}$ are defined by
    \begin{equation*}\label{eq:beta-psi}
    \mkern-6mu \beta_{\psi_n}^{l,u}(x) \mkern-4mu = \mkern-3mu
    \begin{cases}
        \psi_n(0) - \psi_n(x) + \kappa_n^- x, \mkern-12mu & \text{if } x \in (\eta_n^-, 0], \\
        \psi_n(0) - \psi_n(x) + \kappa_n^+ x, \mkern-12mu & \text{if } x \in [0, \eta_n^+), \\
        \lambda_0, \mkern-12mu & \text{if } x \in [l,u] \backslash (\eta_n^-, \eta^+_n).
    \end{cases}
    \end{equation*}
    where
    \(
    \eta_n^- = \max\{\alpha_n^-, l\}, \quad \eta^+_n = \min\{\alpha_n^+, u\}
    \). Moreover, for $l\neq 0$ and $u\neq0$, the slopes $\kappa_n^-$ and $\kappa_n^+$ are given by
    \[
    \kappa_n^- = 
    \begin{cases}
        \psi_n'(\alpha_n^-), & \text{if } \alpha^-_n \geq l, \\
        l^{-1} \left( \lambda_0 + \psi_n(l) - \psi_n(0) \right), & \text{if } \alpha^-_n < l,
    \end{cases}
    \]
    \[
    \kappa_n^+ = 
    \begin{cases}
        \psi_n'(\alpha_n^+), & \text{if } \alpha^+_n \le u, \\
        u^{-1} \left( \lambda_0 + \psi_n(u) - \psi_n(0) \right), & \text{if } \alpha^+_n > u.
    \end{cases}
    \]
\end{proposition}
\begin{proof}
The proof can be found in Appendix~\ref{proof:main_prop}.   
\end{proof}

 In Figure~\ref{fig:illustration_BR}, we present one-dimensional examples of box-constrained \BR{}. We distinguish two cases. In the first one (left graph), the box-constraint $[l,u]$ contains the interval $[\alpha^-,\alpha^+]$ where the unconstrained \BR{}~\cite{Exact2024} is non-constant. In this case, both box-constrained and unconstrained \BR{} coincide. In the opposite situation where $[l,u] \subset [\alpha^-,\alpha^+]$ (right graph), the box-constrained \BR{} deviates from its unconstrained (showed in gray) counterpart.

Equipped with the box-constrained \BR{}~\eqref{eq:continuous_form}, we consider the following continuous relaxation of $J_0$ for $\x \in [l,u]^N$,
\begin{equation}  \label{eq:relaxed}
    J^{l,u}_\Psi(\x) = F_\y(\A \x) + B_\Psi^{l,u}(\x) + \frac{\lambda_2}{2} \|\x\|^2_2.
\end{equation}

\section{Exact Relaxations Properties}

In Theorem~\ref{th:exact_relax}, we provide conditions on $\Psi$ so that $J_\Psi$ is an \textit{exact relaxation} of $J_0$,  meaning it preserves  global minimizers while eliminating certain local ones.

\begin{theorem}[Exact relaxation property]\label{th:exact_relax}
 Let $\Psi$ be such that, $\forall n \in [N]$ and $\forall t \in (\eta^-_n,0) \cup (0,\eta^+_n) $
 \begin{equation}\label{eq:cc}
     \frac{\partial^2}{\partial t} F_\y(\A(\x^{(n)} + t \e_n)) + \lambda_2 < \psi''_n(t) , 
 \end{equation}
 where $\x^{(n)} = (x_1,\ldots,x_{n-1},0,x_{n+1},\ldots,x_N)^T$. Then,
 \begin{align}
     & \argmin_{\x \in [l,u]^N} \,  J^{l,u}_\Psi(\x) = \argmin_{\x \in [l,u]^N} \,  J_0(\x), \\
     & \hat{\x} \text{ local minimizer of } J^{l,u}_\Psi \text{ over } [l,u]^N \, \notag \\ 
     & \qquad \qquad  \Longrightarrow \,  \hat{\x} \text{ local minimizer of } J_0 \text{ over } [l,u]^N.
 \end{align}
\end{theorem}
\begin{proof}
     This result is a direct generalization of~\cite[Theorem 9]{Exact2024}. It relies on three facts: i) $J^{l,u}_\Psi(\x) \leq J_0(\x)$  $\forall \x \in [l,u]^N$ (by definition of $ B_\Psi^{l,u}$), ii) $\forall \x \notin \prod_{n\in [N]} (\eta_n^-,0) \cup (0,\eta_n^+)$, $J^{l,u}_\Psi(\x) = J_0(\x)$  (from Proposition~\ref{prop:main_prop}) and iii) under~\eqref{eq:cc} $\forall n \in [N]$, $t \mapsto J^{l,u}_\Psi(\x^{(n)} + t\e_n )$ is strictly concave on both $(\eta_n^-,0)$ and $(0,\eta_n^+)$.
\end{proof}

\section{Minimizing $J_\Psi^{l,u}$} \label{sec:minminiz_JPsi}

In this section, we discuss the optimization of the relaxed problem through two different optimization algorithms: the forward backward splitting (FBS) \cite{CombettesWajs2005}
 and the iteratively reweighted $\ell_1$ (IRL1) \cite{OchsIRL1}.

\subsection{Forward-Backward Splitting Algorithm} \label{sec:FBS}

The FBS or proximal gradient algorithm consists of the following time stepping scheme for a given initialisation $\x^0$ and step-size $\rho>0$:
\begin{equation*}
    \x^{k+1} \in \prox_{\rho ( B_\Psi^{l,u} + \mathcal{I}_{[l,u]^N}) } \left( \x^k - \rho (  \A^T \nabla F_\y(\A\x^k) + \lambda_2 \x^k) \right).
\end{equation*}
It requires the ability to efficiently evaluate the proximal operator of $B_\Psi^{l,u} + \mathcal{I}_{[l,u]^N}$. Since \BR{} is a separable penalty, this computation reduces to evaluate independent one-dimensional proximal operators, for which a closed-form expression can be derived (for standard choices of $\Psi$) from the following proposition.
\begin{proposition}\label{prop:prox}
    Let $\rho > 0$ and $n \in [N]$. For $x \in \R$, the proximal operator of $\rho \beta^{l,u}_{\psi_n}$ is given by
\begin{equation}\label{eq:prox-box}
\prox_{\rho \beta^{l,u}_{\psi_n} }(x) =\argmin_{v \in [l,u] \cap \mathcal{V}(x)} \left\lbrace \beta^{l,u}_{\psi_n}(v) + \frac{1}{2\rho } (v - x)^2\right\rbrace
\end{equation}
where \(\mathcal{V}(x) = \{l, 0, x, u\} \cup S_x\) with \(S_x = \{v \in \mathbb{R} : v - \rho\psi_n'(v) = x - \rho \kappa^\pm_n\) \}.
\end{proposition}
\begin{proof}
The proof can be found in Appendix~\ref{app:proof-prox}.
\end{proof}
The functional $J_\Psi^{l,u}$ satisfies the Kurdyka-Łojasiewicz property, hence taking $0<\rho< 1/L$, with $L$ being the Lipschtiz constant of the gradient of $F_\y(\A \cdot) + \lambda_2/2 \|\cdot\|_2^2$, the  sequence $\{\x^k\}_k$ generated by FBS converge to a critical point of $J_\Psi^{l,u}$~\cite{Attouch2013}. Alternatively, a backtracking strategy can be used (see Section \ref{sec:experiments}) to estimate the step-size at each iteration, which helps improving the convergence speed.

\subsection{Iteratively Reweighted $\ell_1$}  \label{sec:IRL1}

The Iteratively Reweighted $\ell_1$ (IRL1) algorithm belongs to the class of majorization-minimization (MM) algorithms, which iteratively construct minimizing sequences of surrogate functions that upper-bound the original objective function. The optimization process consists of two main steps. First, in the \textit{majorization step}, the objective function is upper-bounded by a surrogate equal to it at the current point.  For symmetric \BR{} (i.e., $\beta_{\psi_n}^{l,u}(x) = \beta_{\psi_n}^{l,u}(|x|) \; \forall (x,n) \in [l,u] \times [N]$, as those we use in our experiments), we consider the following weighted $\ell_1$-norm
\begin{equation*} 
\tilde J(\x) = F_\y(\A\x) + \sum_{n=1}^N w_n |x_n| + \frac{\lambda_2}{2}\|\x\|^2_2
, 
\end{equation*} where the weights $\{w_n\}_n$ are such that
\( w_n \in \partial \beta_{\psi_n}^{l,u}(|x_n|).\)
Then, the \textit{minimization step} consists of minimizing $\tilde J$ over the set $[l,u]^N$. This is achieved using the Projected FBS algorithm by simply considering the (closed-form) proximal operator of
\(\x \mapsto \sum_{n=1}^N w_n |x_n| + \frac{\lambda_2}{2}\|\x\|^2_2\). Again as the Kurdyka-
Łojasiewicz property holds for $J_\Psi^{l,u}$, convergence of the iterative scheme to a critical point can be shown following \cite{OchsIRL1}.

\section{Experiments} \label{sec:experiments}

\subsection{Data Generation}\label{sec:data-generation}
\textit{Forward Matrix Generation}: We generate a the matrix \( \A \sim \mathcal{N}(\textbf{0}, \bm{\Sigma}) \) of size \( 500 \times 1000 \) from a multivariate normal distribution with mean zero and covariance matrix \( \bm{\Sigma} \). The covariance \( \bm{\Sigma} \) follows an exponential correlation model, where its entries are defined by \( \sigma_{ij} = \rho^{|i-j|} \) for \( 1 \leq i, j \leq N \), with the parameter \( \rho \in [0,1] \) controlling the correlation strength. We then consider two differents settings to generate the observation vector \( \y \in \R^M \) according to the data term $F_{\y}$ of interest.
\begin{itemize}
    \item \textit{Least-Square (LS)}, $F_\y(\A\x)=\frac{1}{2}\|\A\x-\y\|^2$
    : following~\cite{guyard:24}, we define a sparse vector \( \x^* \in \R^N \) with \( k^* \in \N \) non-zero equispaced  entries, each sampled from a uniform distribution in the interval \( [l, u] \). The observation $\y$ is then generated by \( \y = \A \x^* + \bm{\varepsilon} \), where \( \varepsilon_i \sim N(0, \sigma^2) \)
    . The signal-to-noise  ratio (SNR) measure defined as
    \[
    \mathrm{SNR} = \frac{\operatorname{Var}(\A \x^*)}{\operatorname{Var}(\bm{\varepsilon})} = \frac{({\x^*})^T \bm{\Sigma} \x^*}{\sigma^2}
    \]
    is used to control the level of noise on the data $\y$.
    \item \textit{Logistic Regression (LR)}, $F_\y(\A\x)=\sum_{m=1}^M \log\left(1+\exp({[\A\x]_m)} \right)-y_m[\A\x]_m$: As for LS, we generate a \( k^* \)-sparse vector \( \x^* \) with equispaced non-zero entries, each equal to 1. Each coordinate of the label vector \( \y \) is binary, with \( y_m \in \{-1, 1\} \), where \( y_m = 1 \) is determined with  probability
    \[
    P\left(y_m = 1 \mid \a_m \right) = 1/\left( 1 + \e^{-s\left< \a_m, \x^* \right>}\right),
    \]
    where \( \a_m \) is the \( m \)-th row of \( \A \) and \( s > 0 \) controls the SNR. The  labels $\y$ are  sampled from a Bernoulli distribution.
\end{itemize}

\subsection{Algorithmic Setting and Parameter Selection}

\paragraph*{Benchmarked Algorithms} We compare the performance of the following four methods in solving Problem~\eqref{eq:problem_setting} for 20 realizations of forward matrices $\A$ and vectors $\y$. 
\begin{itemize}
     \item \textit{IHT}~\cite{Attouch2013} on the original problem,
    \item \textit{FBS} on the proposed exact relaxation (Section \ref{sec:FBS}) 
        \item \textit{IRL1} on the proposed exact relaxation (Section \ref{sec:IRL1}) 
            \item \textit{BnB}~\cite{guyard:24} on the original problem. It can guarantee the convergence to a global minimizer for small-scale problems.

\end{itemize}

\textit{Parameters}:
For LS problems, initialization was set as $\x_0=\textbf{0}$   for all algorithms. The same choice was also made for LR problems for all algorithms except for 
IHT, where the initialization $\x_0=\A^T\y$ was considered. Indeed, as shown in~\cite[Lemma 1]{Exact2024}, $\x_0=\textbf{0}$ is a local minimizer of $J_0$ (note that $\bm{0}$ is, in contrast  and algorithms that minimize $J_0$ (as IHT) may thus be stuck at the initial point $\x_0 = \bm{0}$. As for FBS,  we implemented IHT with a backtracking strategy to accelerate convergence and to favour (upon large initial time steps) escape from the local minimizer $\x_0=\textbf{0}$ for the LS case. As this turned out to be harder for LR, a different initialization was chosen instead. 
We fixed the convergence tolerance to a relative change between consecutive iterates below  $10^{-7}$. 

The hyperparameter \( \lambda_0 \) is set as \( \lambda_0 = \alpha F_\mathbf{y}(\mathbf{0}) \), with $\alpha \in (0,1)$ chosen so that the BnB solver provides a solution with a support cardinality close to the ideal \( k^* \), in all experiments. For the LS problem, we set \( \lambda_2 = 0 \), focusing purely on sparsity, while for the LR problem we used \( \lambda_2 =1 \). 
As far as the generating functions $\psi_n $ are considered, we considered in the following $\psi_n(x) = (\gamma_n/2) x^2$, with $\gamma_n = \lambda_2 + \|\a_n\|^2_2$ for LS and $\gamma_n = \lambda_2 + 0.25 \|\a_n\|^2_2 $ for LR. This choice ensures that the exact relaxation condition~\eqref{eq:cc} is satisfied.

\subsection{Results and Discussion}

\begin{figure*}[!ht]
    \centering
    \begin{subfigure}[b]{0.48\linewidth}
        \centering
        \includegraphics[width=\linewidth]{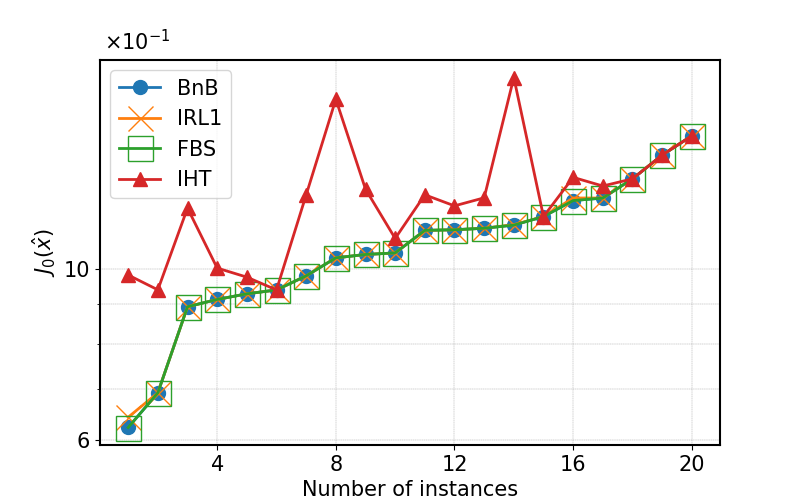}
    \end{subfigure}
    \begin{subfigure}[b]{0.48\linewidth}
        \centering
        \includegraphics[width=\linewidth]{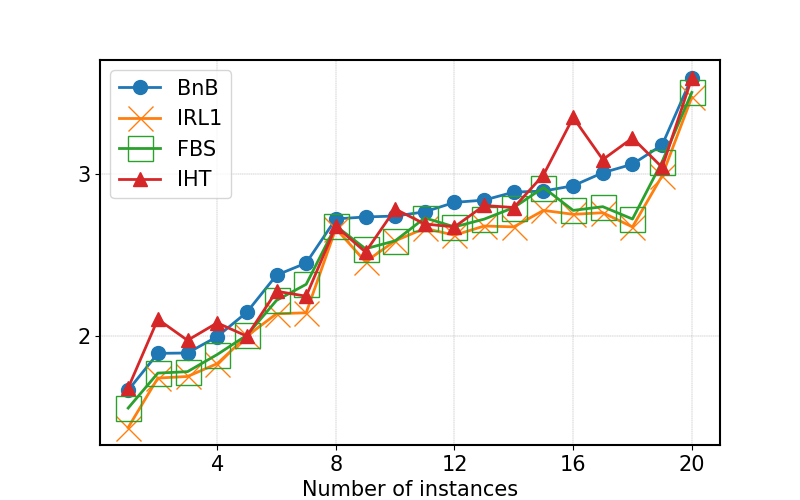}
    \end{subfigure}
    \caption{LS: (ordered) values $J_0(\hat\x)$ obtained by each method along the $20$ problem instances. For the left plot $\lambda_0 = 2 \times 10^{-2} F_\y(\textbf{0})$ and $k^*=10$. For the right plot $\lambda_0 = 5 \times 10^{-3}F_\y(\textbf{0})$ and $k^*=25$. Finally, $[l,u] = [-1.5, 1.5]$, $\rho=0.9$ and $\mathrm{SNR} = 10$. \label{fig:Ls-resutls}}
\end{figure*}

\begin{figure*}[!ht]
    \centering
    \begin{subfigure}[b]{0.48\linewidth}
        \centering
        \includegraphics[width=\linewidth]{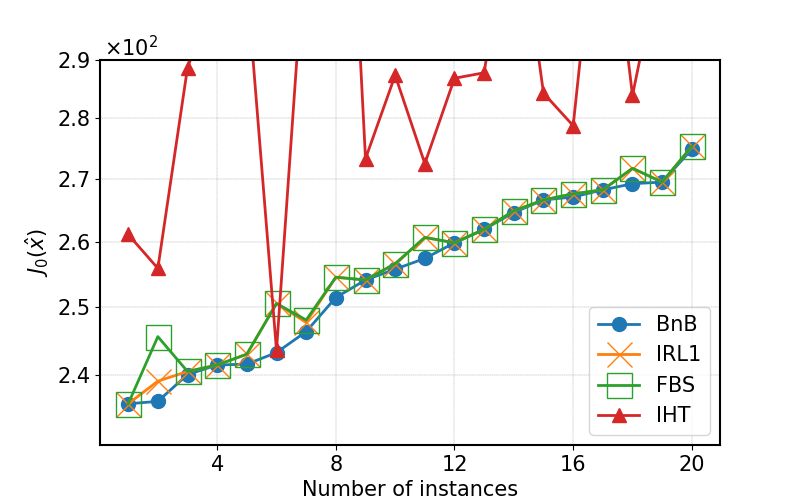}
    \end{subfigure}
    \begin{subfigure}[b]{0.48\linewidth}
        \centering
        \includegraphics[width=\linewidth]{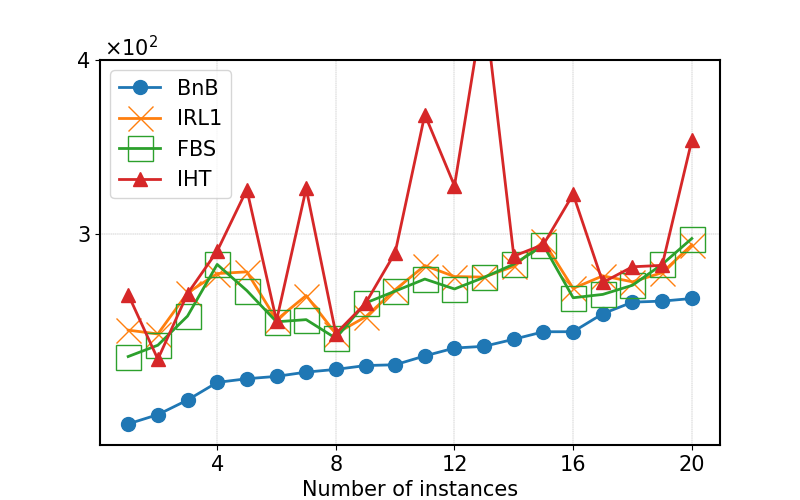}
    \end{subfigure}
    \caption{LR: (ordered) values $J_0(\hat\x)$ obtained by each method along the $20$ problem instances. For the left plot $\lambda_0 = 2.5 \times 10^{-2} F_\y(\textbf{0})$ and $k^*=7$. For the right  $\lambda_0 = 1.5 \times 10^{-2} F_\y(\textbf{0})$ and $k^*=25$. Finally $[l,u] = [-1, 1]$, $\lambda_2 = 1$, $\rho = 0.9$, $s=1$. \label{fig:Lr-resutls}}
\end{figure*}

We assess the quality of the  minimization using the value of the original function $J_0$ at convergence. For each algorithm, we thus computed this value for each realization of $(\A, \y)$. 

\textit{BnB with certification of the global solution:} The left panels of Figures~\ref{fig:Ls-resutls} and~\ref{fig:Lr-resutls} present results for  LS and LR problems when $k^* = 10$ and $k^* = 7$, respectively. In these cases, the BnB solver consistently finds (and certifies) the global optimum. From these results we can thus observe that the minimization of the proposed exact relaxation with either IRL1 or FBS also often leads to the global minimizer. In contrast, this is not the case when directly tackling the original problem with IHT.

\textit{BnB without certified solutions:} In the right panels of Figures~\ref{fig:Ls-resutls} and \ref{fig:Lr-resutls}, we consider the case of less sparse vectors ($k^* = 25$), where BnB fails to certify optimality within the given time limit (30 minutes, whereas FBS and IRL1 solve the relaxed problem in an average time of 2.9s/3.87s for the LS cases and 0.63s/0.62s for the LR cases, respectively). In the LS setting, BnB appears to be far from the global optimum, as minimizing the relaxed problem yields solutions with a lower objective function value. Conversely, in the LR setting, while BnB does not certify optimality, it still finds solutions with the lowest objective function value, followed by methods minimizing the proposed relaxations. 
The difference in performance of the BnB between LS and LR can be attributed to the presence of the $\ell_2$ term in the case of LR ($\lambda_2 >0$).
 As shown in~\cite{guyard:24}, the incorporation of the $\ell_2$ term improves the relaxation quality for pruning tests, which explains why BnB performs better in the LR setting. This is further supported by the additional experiment we provide in Appendix~\ref{app:additional-exp}. There, we set $\lambda_2 >0$ in the LS experiment with $k^* =25$ and one can observe that the BnB performs better than the other methods.

\section{Conclusion}
We introduced the constrained $\ell_0$ Bregman relaxation  which continuously approximate the $\ell_0$ pseudo-norm within a bounded domain $[l, u]^N$. By replacing the $\ell_0$ pseudo-norm with this constrained \BR{}, a continuous optimization problem that preserves the same global minimizers as the original one and exhibit fewer local minimizers is obtained.
This relaxation makes the problem more amenable to standard non-convex optimization algorithms, such as the proximal gradient method and the iteratively reweighted $\ell_1$ (IRL1) algorithms. Through several numerical experiments, we demonstrated that our approach compares well to Branch-and-Bound (BnB) approaches, achieving strong agreement with certified global solutions in small-scale problem settings.

\appendix 

\subsection{Proof of Proposition~\ref{prop:main_prop}}\label{proof:main_prop}

The separability of \( B_\Psi^{l,u} \) follows from the separability of the Bregman divergence, the \(\ell_0\) pseudo-norm, and the indicator function. We thus provide the proof in the one-dimensional case. Let \( (x,n) \in [l,u] \times [N]\). For \(z \in \mathrm{dom}(\psi_n)\), we have that \(\alpha - d_{\psi_n}(\cdot, z) \leq \lambda_0 |\cdot|_0 + \mathcal{I}_{[l,u]}(\cdot)\) if and only if \(\alpha \leq d_{\psi_n}(\cdot, z) + \lambda_0 |\cdot|_0 + \mathcal{I}_{[l,u]}(\cdot)\). Thus, the supremum with respect to \(\alpha\) in the definition of \BR{} for \(z \in \mathrm{dom}(\psi_n)\) is given by
\begin{equation}\label{eq:alpha-formula}
   \alpha(z) = \inf_{x \in \mathrm{dom}(\psi_n)} \lambda_0 |x|_0 + d_{\psi_n}(x,z) + \mathcal{I}_{[l,u]}(x).
\end{equation}
We distinguish two cases for solving Problem~\eqref{eq:alpha-formula}:

\begin{itemize}
    \item If \( z\in [l, u]\), since \(\psi_n\) is strictly convex, the function  \(d_{\psi_n}(\cdot, z)\) is also strictly convex for any~\(z\). Consequently, the function
   \(
   x \mapsto \lambda_0 |x|_0 + d_{\psi_n}(x, z) + \mathcal{I}_{[l,u]}(x)
   \)
   has two local minimizers: one at \(x = 0\) with value \(d_{\psi_n}(0, z)\) and one at \(x = z\) with value \(\lambda_0\). Thus, we get
\begin{align}\label{eq:1}
   \alpha(z) &= \min(\lambda_0, d_{\psi_n}(0, z)) \notag\\
             & = 
   \begin{cases}
       d_{\psi_n}(0, z), & \text{if } z \in [\alpha_n^-, \alpha_n^+], \\
       \lambda_0, & \text{otherwise}.
   \end{cases}
\end{align}
where \([\alpha_n^-, \alpha_n^+]\) is the interval of the \(\lambda_0\)-sublevel set of the function $d_{\psi_n}(0,z)$ 
\item If \( z \notin [l, u] \), the minimum is reached either at \(x = 0\) with value \(d_{\psi_n}(0, z)\) or at \(x = u\) (or \(x = l\)) with value \(\lambda_0 + d_{\psi_n}(u, z)\) (or \(\lambda_0 + d_{\psi_n}(l, z)\), respectively). Note that the choice \(x = u\) or \(x = l\) depends  on the sign of \(z\). We thus focus on the non-negative case $x=u$. We have
   \begin{equation}
       \alpha(z) = \min(\lambda_0 + d_{\psi_n}(u, z), d_{\psi_n}(0, z)).
   \end{equation}
First, let us compare the quantities \( d_{\psi_n}(0, z) \) and \( \lambda_0 + d_{\psi_n}(u,z) \). For $z \geq 0$, we get:
\begin{align*}
   d_{\psi_n}(0,z) &\leq \lambda_0 + d_{\psi_n}(u,z) \\&
   \implies z \leq (\psi_n')^{-1}\left(\frac{ \lambda_0 + \psi_n(u)-\psi_n(0)}{u}\right),
\end{align*}
which holds since $\psi_n'$ is strictly increasing (as $\psi_n$ strictly convex).
Similarly for $z\leq 0$, $ d_{\psi_n}(0,z) \leq \lambda_0 + d_{\psi_n}(l,z)$ implies $ z \geq (\psi_n')^{-1}\left(\frac{ \lambda_0 + \psi_n(l)-\psi_n(0)}{l}\right)$.
Let us now denote by 
$\xi^+_n := (\psi_n')^{-1}\left(\frac{ \lambda_0 + \psi_n(u)-\psi_n(0)}{u}\right)$ and $\xi^-_n := (\psi_n')^{-1}\left(\frac{ \lambda_0 + \psi_n(l)-\psi_n(0)}{l}\right)$. We have
\begin{equation}\label{eq:2}
\alpha(z) =   \begin{cases}
       d_{\psi_n}(0, z), & \text{if } z \in [\xi_n^-, \xi_n^+], \\
       \lambda_0 + d_{\psi_n}(u,z), & \text{if } z \geq \xi_n^+ \\
       \lambda_0 + d_{\psi_n}(l,z), & \text{if }  z \leq \xi_n^-.
   \end{cases}
\end{equation}
\end{itemize}

In order to combine the two cases above, we analyse the order of the quantities $\xi^-_n, \xi^+_n, \alpha_n^-$, and $\alpha_n^+$. From the definition of the Bregman divergence, we have $d_{\psi_n}(u,v) \geq 0 \; \forall u,v \in \mathrm{dom}(\psi_n)$.
Hence, 
\begin{align*}
   z \in [\alpha_n^-, \alpha_n^+] &\implies d_{\psi_n}(0,z) \leq \lambda_0 \\
   &\implies 
   \left\lbrace \begin{array}{ll}
        d_{\psi_n}(0,z)  & \leq \lambda_0 + d_{\psi_n}(u,z)  \\
        d_{\psi_n}(0,z)  & \leq \lambda_0 + d_{\psi_n}(l,z) 
   \end{array} \right.\\&
   \implies z \in [\xi^-_n, \xi^+_n].
\end{align*}
Therefore, we have \(\xi^-_n \leq \alpha_n^- \leq 0 \leq \alpha_n^+ \leq \xi^+_n\).

Now, let \(u \in [0, \xi^+_n]\). Since \(\psi_n'\) is non-decreasing, we have \(\psi_n'(u) \leq \psi_n'(\xi^+_n)\).
From the definition of \(\alpha_n^+\), we have \(\lambda_0 = d_{\psi_n}(0,\alpha_n^+)\), which implies \(\lambda_0 - \psi_n(0) = \alpha_n^+\psi_n'(\alpha_n^+) - \psi_n(\alpha_n^+)\).
By definition of \(\xi^+_n\), we have \(\psi_n'(\xi^+_n) = \frac{\lambda_0 + \psi_n(u) - \psi_n(0)}{u}\). Thus, we get
\[
u\psi_n'(\xi^+_n) = \alpha_n^+\psi_n'(\alpha_n^+) - \psi_n(\alpha_n^+) + \psi_n(u).
\]
Thus,
\begin{align*}
   u\psi_n'(u) &\leq u\psi_n'(\xi^+_n) = \alpha_n^+\psi_n'(\alpha_n^+) - \psi_n(\alpha_n^+) + \psi_n(u) \\
   &\implies u\psi_n'(u) - \psi_n(u) \leq \alpha_n^+\psi_n'(\alpha_n^+) - \psi_n(\alpha_n^+) \\&
   \implies u \leq \alpha_n^+,
\end{align*}
where the last inequality comes from the fact that \( g(z) = z\psi_n'(z) - \psi_n(z) \) is non-decreasing for \( z \geq 0 \). Since \(\psi_n\) is twice differentiable, \(g(z)\) admits a derivative given by \( g'(z) = z\psi_n''(z) \). The convexity of \(\psi_n\) implies that \(\psi_n''(z) \geq 0\), hence \( g'(z) \geq 0 \), which shows that \( g(z) \) is indeed non-decreasing for \( z \geq 0 \). Similarly, for $l \in [\xi^-_n, 0]$, we can show that $l \geq \alpha_n^-$. Hence there are only two possibilities for $u$ (resp., $l$): either $u \leq \alpha_n^+$ (resp., $l \geq \alpha_n^-$), or $u \geq \xi_n^+$ (resp., $l \leq \xi_n^-$).

Combining this with~\eqref{eq:1} and~\eqref{eq:2}, we obtain, for \( z \geq 0 \):
\begin{itemize}
   \item If \( u \leq  \xi^+_n \) (and so $u \leq \alpha_n^+$),  we have
   \begin{equation}\label{eq:case1-sup}
       \alpha(z) = \begin{cases}
           d_{\psi_n}(0,z), & \text{if } z \in [0, \xi^+_n] \\
           \lambda_0 + d_{\psi_n}(u,z), & \text{otherwise}
       \end{cases}
   \end{equation}
   \item If \( u \geq \xi^+_n \),
   \begin{equation}\label{eq:case2-sup}
       \alpha(z) = \begin{cases}
           d_{\psi_n}(0,z), & \text{if } z \in [0, \alpha_n^+] \\
           \lambda_0, & \text{if } z \in [\alpha_n^+, u] \\
           \lambda_0 + d_{\psi_n}(u,z), & \text{if } z \geq u .
       \end{cases}
   \end{equation}
\end{itemize}
We have similar cases for \( z \leq 0 \), by considering the quantities \(\alpha_n^-\), \(l\), and \(\xi^-_n\).

To complete the proof, we need to compute the supremum with respect to \(z\):
\begin{equation}\label{eq:sup-z}
\beta_{\psi_n}^{l,u}(x) = \sup_{z \in \R} \left( \alpha(z) - d_{\psi_n}(x, z) \right).
\end{equation}
We distinguish several cases based on equations~\eqref{eq:case1-sup} and~\eqref{eq:case2-sup}:

\begin{itemize}
   \item  If \( u \leq  \xi^+_n \) (and so $u \leq \alpha_n^+$) and $x \in [0,u]$, from~\eqref{eq:case1-sup}, we have the following two cases:
   
       - For \(z \in [0, \xi_n^+]\), we have:
       \[
         \sup_{z \in [0, \xi_n^+]} \left( d_{\psi_n}(0, z) - d_{\psi_n}(x, z) \right),
       \]
       which, from the definition of \(d_{\psi_n}\), simplifies to:
       \begin{align*}
             & \mkern-8mu \psi_n(0) - \psi_n(x) + \sup_{z \in [0, \xi_n^+]} \left( - \psi_n'(z) (0 - z) + \psi_n'(z) (x - z) \right) \\
           = & \psi_n(0) - \psi_n(x) + \sup_{z \in [0, \xi_n^+]} \left( \psi_n'(z) x \right).
       \end{align*}
       Since \(\psi_n\) is strictly convex, \(\psi_n'\) is increasing, the supremum is attained at \(z = \xi_n^+\). Therefore, we have:
       \begin{equation}\label{eq:case1-sup-z}
           b_1(x) := \psi_n(0) - \psi_n(x) + \psi_n'( \xi_n^+ ) x,
       \end{equation}
       
       - For \(z \geq \xi_n^+\), we have:
       \[
       \sup_{z \geq \xi_n^+} \left( \lambda_0 + d_{\psi_n}(u, z) - d_{\psi_n}(x, z) \right),
       \]
       which, from the definition of \(d_{\psi_n}\) and the fact that $x \leq u$, simplifies to:
\begin{align*}
    \; & \lambda_0 + \psi_n(u) - \psi_n(x)  
   + \sup_{z \geq \xi_n^+} \left( \psi_n'(z) (x - u) \right) \\& 
   = \; \lambda_0 + \psi_n(u) - \psi_n(x) + \psi_n'(\xi_n^+) (x - u) 
   :=  \;b_2(x).
\end{align*}

       
\item If \( u \geq \xi^+_n \geq \alpha_n^+ \), and $x \in [0,u]$, then we have the following cases from~\eqref{eq:case2-sup}:

   - For \(z \in [0, \alpha_n^+]\), we have:
   \[
   \sup_{z \in [0, \alpha_n^+]} \left( d_{\psi_n}(0, z) - d_{\psi_n}(x, z) \right),
   \]
   which, from the definition of \(d_{\psi_n}\), simplifies to:
   \begin{align*}
       \psi_n(0) - \psi_n(x) + \sup_{z \in [0, \alpha_n^+]} \left( \psi_n'(z) x \right).
   \end{align*}
   Since \(\psi_n\) is strictly convex, \(\psi_n'\) is increasing, the supremum is attained at \(z = \alpha_n^+\). Therefore, we have:
   \begin{equation}\label{eq:case3-sup-z}
       b_3(x) := \psi_n(0) - \psi_n(x) + \psi_n'(\alpha_n^+) x,
   \end{equation}
   - For \(z \in [\alpha_n^+, u]\), we have:
   \[
   \sup_{z \in [\alpha_n^+, u]} \left( \alpha(z) - d_{\psi_n}(x, z) \right) = \sup_{z \in [\alpha_n^+, u]} \left( \lambda_0 - d_{\psi_n}(x, z) \right),
   \]
   which results in:
   \begin{equation}\label{eq:case5-sup-z}
       b_4(x) := \begin{cases}
           \lambda_0 - d_{\psi_n}(x, \alpha_n^+), & \text{if } x \leq \alpha_n^+, \\
           \lambda_0, & \text{if } x \in [\alpha_n^+, u], 
       \end{cases}
   \end{equation}

      - For \(z \geq u\), since $x \leq u$, the supremum is attained at $\alpha_n^+$, given by $b_5(x) = \lambda_0 - d_{\psi_n}(x, \alpha_n^+)$.
\end{itemize}

As a conclusion, we have:

\begin{itemize}
   \item If \( u \leq \xi_n^+ \)   (so is $u \leq \alpha_n^+$. Since \(b_2(x)=\lambda_0 + \psi_n(u) - \psi_n(x) + \psi_n'(\xi_n^+)(x - u) = \psi_n(0) - \psi_n(x) + \psi_n'(\xi_n^+) x = b_1(x)\) (by definition of $\xi_n^+$), which mean that $\beta_{\psi_n}^{l,u}(x) = \max \left(b_1(x), b_2(x) \right) = b_1(x)$. Hence, we have that
   \begin{equation*}\label{eq:beta-formula1}
   \beta_{\psi_n}^{l,u}(x) = 
       \psi_n(0) - \psi_n(x) + \psi_n'(\xi_n^+)x \quad \forall  x \in [0,u]
   \end{equation*}
   
   \item If \( u \geq \xi_n^+ \), we have
   \begin{align}
   \beta_{\psi_n}^{l,u}(x) &= \max\left(b_3(x), b_4(x), b_5(x) \right) \notag \\
              &= \begin{cases}
                   \psi_n(0) - \psi_n(x) + \psi_n'(\alpha_n^+) x, & \text{if } x \in [0, \alpha_n^+], \\
                   \lambda_0, & \text{if } x \in [\alpha_n^+, u].
                 \end{cases}
                 \notag
   \end{align}

   This follows from the fact that \(\lambda_0 - d_{\psi_n}(x, \alpha_n^+) = \psi_n(0) - \psi_n(x) + \psi_n'(\alpha_n^+) x\) (by definition of \(\alpha_n^+\)). 

\end{itemize}
Performing similarly for \( x \in [l,0]\), we thus complete the proof.

\subsection{Proof of Proposition~\ref{prop:prox}}\label{app:proof-prox}

The proof follows from the fact that (see~\cite[Section 3.2]{Exact2024}):
\begin{equation*}
    \partial \beta_{\psi_n}^{l, u}(v) = \begin{cases}
          \{-\psi'(v) + \kappa_n^-\} & \text{if } v < 0 \\
          \{-\psi'(v) + \kappa_n^+\} & \text{if } v > 0 \\
         \{-\psi'(0)\} + [\kappa^-_n, \kappa^+_n] & \text{if } v=0
    \end{cases}
\end{equation*}

Recalling that 
\[
     \prox_{\rho \beta^{l,u}_{\psi_n}}(x) = \argmin_{v \in [l,u]} \left\lbrace \beta^{l,u}_{\psi_n}(v) + \frac{1}{2\rho} (v - x)^2 \right\rbrace,
\]
the first-order optimality condition states that
\[
    0 \in \frac{1}{\rho}(v-x) + \partial \beta_{\psi_n}^{l,u}(v) + \mathcal{N}_{[l,u]}(v),
\]
where $\mathcal{N}_{[l,u]}$ is the normal cone\footnote{The normal cone of a set $\Cc \subset \mathbb{R}$ is defined as: 
$
    \mathcal{N}_\Cc(x) = \left\lbrace t \in \mathbb{R} \; | \; \left\langle t, z-x \right\rangle \leq 0 \; \forall z \in \Cc \right\rbrace
$.
} to  $[l,u]$, defined as
\[
    \mathcal{N}_{[l,u]}(v) = 
    \begin{cases} 
        \{0\} & \text{if } v \in (l,u), \\
        t \geq 0 & \text{if } v = u, \\
        t \leq 0 & \text{if } v = l.
    \end{cases}
\]
It follows that the possible solutions of the proximal operator are included in
\(
    \{0, x, l, u\} \cap S_x,
\)
where $S_x$ is the set of solutions of the equation: $-\psi_n'(v) + \kappa^\pm_n + \frac{1}{\rho}(v-x) = 0.
$

\subsection{Influence of the Parameter $\lambda_2$}\label{app:additional-exp}

In our experiments, we focused on pure sparsity-based problems (i.e., $\lambda_2 = 0$) in the least-squares setting. The results indicate that for $k^ = 25$, the BnB solver, within the 30-minute time limit, yields solutions with higher values of the functional $J_0$ compared to those obtained via the minimization of the proposed relaxations or the application of IHT to the original problem (see the left panel of Figure~\ref{fig:Ls-resutls}). In Figure~\ref{fig:enter-label}, we provide additional results for the same instances of $(\A, \y)$, but with $\lambda_2 = 2$. Here, we observe that both the BnB solver and FBS on the proposed relaxation achieve solutions with the lowest objective values. This is in agreement with the fact that it is well-known that considering $\lambda_2>0$ improves pruning tests for BnB methods. Moreover we observe that even in this setting the proposed methods compare favorably  with BnB, yet with a much lower computational cost.

\begin{figure}[!ht]
    \centering
    \includegraphics[width=0.9\linewidth]{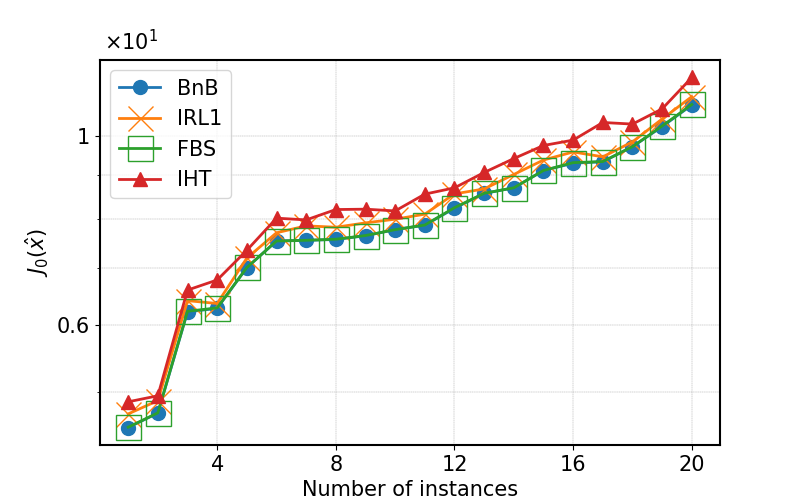}
    \caption{LS: (ordered) values of $J_0(\hat\x)$ obtained by each method along the $20$ problem instances, with $\lambda_0 = 5.5 \times 10^{-3} F_\y(\textbf{0})$, $\lambda_2 = 2$, $[l,u] = [-1.5, 1.5]$, $k^* = 25$, $\rho = 0.9$ and $\mathrm{SNR} = 10$.}
    \label{fig:enter-label}
\end{figure}

\bibliographystyle{IEEEtran}
\bibliography{refs}

\begin{thebibliography}{10}
\providecommand{\url}[1]{#1}
\csname url@samestyle\endcsname
\providecommand{\newblock}{\relax}
\providecommand{\bibinfo}[2]{#2}
\providecommand{\BIBentrySTDinterwordspacing}{\spaceskip=0pt\relax}
\providecommand{\BIBentryALTinterwordstretchfactor}{4}
\providecommand{\BIBentryALTinterwordspacing}{\spaceskip=\fontdimen2\font plus
\BIBentryALTinterwordstretchfactor\fontdimen3\font minus \fontdimen4\font\relax}
\providecommand{\BIBforeignlanguage}[2]{{%
\expandafter\ifx\csname l@#1\endcsname\relax
\typeout{** WARNING: IEEEtran.bst: No hyphenation pattern has been}%
\typeout{** loaded for the language `#1'. Using the pattern for}%
\typeout{** the default language instead.}%
\else
\language=\csname l@#1\endcsname
\fi
#2}}
\providecommand{\BIBdecl}{\relax}
\BIBdecl

\bibitem{cameron2013regression}
A.~C. Cameron and P.~K. Trivedi, \emph{Regression analysis of count data}.\hskip 1em plus 0.5em minus 0.4em\relax Cambridge university press, 2013, no.~53.

\bibitem{buhlmann2011statistics}
P.~B{\"u}hlmann and S.~Van De~Geer, \emph{Statistics for high-dimensional data: methods, theory and applications}.\hskip 1em plus 0.5em minus 0.4em\relax Springer Science \& Business Media, 2011.

\bibitem{Nguyen2019}
T.~T. Nguyen, C.~Soussen, J.~Idier, and E.-H. Djermoune, ``{NP-hardness of $\ell_0$ minimization problems: revision and extension to the non-negative setting},'' in \emph{Proceedings of SAMPTA}, Bordeaux, 2019.

\bibitem{Attouch2013}
H.~Attouch, J.~Bolte, and B.~F. Svaiter, ``Convergence of descent methods for semi-algebraic and tame problems: proximal algorithms, forward--backward splitting, and regularized {G}auss--{S}eidel methods,'' \emph{Mathematical Programming}, vol. 137, no.~1, pp. 91--129, 2013.

\bibitem{Beck-SPS}
A.~Beck and N.~Hallak, ``Proximal mapping for symmetric penalty and sparsity,'' \emph{SIAM Journal on Optimization}, vol.~28, no.~1, pp. 496--527, 2018.

\bibitem{Diego}
D.~D. Donne, M.~Kowalski, and L.~Liberti, ``A novel integer linear programming approach for global l0 minimization,'' \emph{Journal of Machine Learning Research}, vol.~24, no. 382, pp. 1--28, 2023.

\bibitem{guyard:24}
T.~Guyard, C.~Herzet, C.~Elvira, and A.-N. Arslan, ``A new branch-and-bound pruning framework for l0-regularized problems,'' in \emph{International Conference on Machine Learning (ICML)}.\hskip 1em plus 0.5em minus 0.4em\relax PMLR, 2024.

\bibitem{Weigeneral}
W.~Bian and X.~Chen, ``{A Smoothing Proximal Gradient Algorithm for Nonsmooth Convex Regression with Cardinality Penalty},'' \emph{SIAM Journal on Numerical Analysis}, vol.~58, no.~1, pp. 858--883, 2020.

\bibitem{carlsson2019convex}
M.~Carlsson, ``{On Convex Envelopes and Regularization of Non-convex Functionals Without Moving Global Minima},'' \emph{Journal of Optimization Theory and Applications}, vol. 183, no.~1, pp. 66--84, 2019.

\bibitem{soubies2015continuous}
E.~Soubies, L.~Blanc-F\'eraud, and G.~Aubert, ``{A Continuous Exact $\ell_0$ Penalty (CEL0) for Least Squares Regularized Problem},'' \emph{SIAM Journal on Imaging Sciences}, vol.~8, no.~3, pp. 1607--1639, 2015.

\bibitem{Exact2024}
M.~Essafri, L.~Calatroni, and E.~Soubies, ``{Exact Continuous Relaxations of $\ell_0$-Regularized Criteria with Non-quadratic Data Terms},'' \emph{arXiv:2402.06483}, 2024.

\bibitem{Essafri2024MLSP}
------, ``{On $\ell_0$ Bregman-Relaxations for Kullback-Leibler Sparse Regression},'' in \emph{2024 IEEE 34th International Workshop on Machine Learning for Signal Processing (MLSP)}, 2024, pp. 1--6.

\bibitem{Lazzaretti2021}
M.~Lazzaretti, L.~Calatroni, and C.~Estatico, ``{Weighted-{CEL}0 sparse regularisation for molecule localisation in super-resolution microscopy with Poisson data},'' in \emph{2021 IEEE 18th International Symposium on Biomedical Imaging (ISBI)}, 2021, pp. 1751--1754.

\bibitem{supp}
``Supplementary material,'' {\url{https://www.irit.fr/~Emmanuel.Soubies/Papiers/SuppMaterialBoxBrex.pdf}}.

\bibitem{CombettesWajs2005}
P.~L. Combettes and V.~R. Wajs, ``Signal recovery by proximal forward-backward splitting,'' \emph{Multiscale Modeling \& Simulation}, vol.~4, no.~4, pp. 1168--1200, 2005.

\bibitem{OchsIRL1}
P.~Ochs, A.~Dosovitskiy, T.~Brox, and T.~Pock, ``On iteratively reweighted algorithms for nonsmooth nonconvex optimization in computer vision,'' \emph{SIAM Journal on Imaging Sciences}, vol.~8, no.~1, pp. 331--372, 2015.

\end{thebibliography}

\end{document}